\documentclass[11pt,leqno]{amsart}
\usepackage{verbatim,amssymb,amsfonts,latexsym,amsmath,amsthm,tikz,bbm,nicefrac,xspace,graphicx,mathtools,mathdots}
\usetikzlibrary{hobby}

\newtheorem{theorem}{Theorem}[section]
\newtheorem{lemma}[theorem]{Lemma}

\newtheorem*{maintheorem}{Main Theorem}

\newtheorem*{theorem*}{Theorem}
\newtheorem*{corollary*}{Corollary}
\newtheorem*{claim1}{Claim 1}
\newtheorem*{claim2}{Claim 2}
\newtheorem*{claim3}{Claim 3}
\newtheorem*{claim4}{Claim 4}
\newtheorem*{claim5}{Claim 5}
\newtheorem*{claim}{Claim}
\newtheorem*{sub-claim}{sub-claim}
\theoremstyle{definition}

\newtheorem{definition}[theorem]{Definition}

\theoremstyle{remark}

\newcommand{\R}{\mathbb{R}}
\newcommand{\N}{\mathbb{N}}

\newcommand{\C}{\mathcal{C}}

\newcommand{\F}{\mathcal{F}}
\newcommand{\G}{\mathcal{G}}

\renewcommand{\P}{\mathcal{P}}

\newcommand{\explicitSet}[1]{\left\lbrace #1 \right\rbrace}

\newcommand{\set}[2]{\explicitSet{#1 \colon #2}}

\newcommand{\e}{\varepsilon}

\newcommand{\w}{\omega}
\newcommand{\0}{\emptyset}
\newcommand{\sub}{\subseteq}

\newcommand{\card}[1]{\left\lvert #1 \right\rvert}

\begin{document}

\title{Three conditionally convergent series}
\author{Will Brian}
\address {
Will Brian\\
Department of Mathematics and Statistics\\
University of North Carolina at Charlotte\\
9201 University City Blvd.\\
Charlotte, NC 28223}
\email{wbrian.math@gmail.com}

\begin{abstract}
It is proved that given any three conditionally convergent series of real numbers, there is a single sequence of natural numbers such that each of the corresponding three subseries sums to either $\infty$ or $-\infty$. An example is provided to show that the analogous statement for four series is false.
\end{abstract}

\maketitle

\section{Introduction}

Recall that an infinite series $\sum_{n \in \N}a_n$ is \emph{conditionally convergent} if it converges, but the series $\sum_{n \in \N}|a_n|$ does not. Every conditionally convergent series has a subseries summing to $\infty$, for example the subseries obtained by summing only over positive terms, and a different subseries summing to $-\infty$, for example the subseries obtained by summing only over negative terms. A conditionally convergent series also has many subseries that diverge by oscillation; for example, one may construct such a subseries by interleaving long stretches of positive terms with long stretches of negative terms.

Let us say that a set $A \sub \N$ sends a series $\sum_{n \in \N}a_n$ of real numbers \emph{to infinity} if the subseries $\sum_{n \in A}a_n$ consisting only of those terms with index in $A$ sums either to $\infty$ or to $-\infty$.

\begin{maintheorem}\label{thm:threeseries}
For any three conditionally convergent series, there is some $A \sub \N$ sending all three series to infinity.
\end{maintheorem}

Notice that we do not require that all three of our subseries sum to $\infty$, or that they all sum to $-\infty$. This may be impossible, even for just two series $\sum_{n \in \N}a_n$ and $\sum_{n \in \N}b_n$, for example if $a_n = -b_n$ for all $n$. However, we do require that each subseries is made to diverge either to $\infty$ or to $-\infty$, and not merely to diverge by oscillation.

Section~\ref{sec:main} contains a proof of the main theorem.
Section~\ref{sec:example} contains an example of four conditionally convergent series such that no single $A \sub \N$ sends all four series to infinity. Thus the main theorem cannot be improved by replacing three with four series.


The ideas in this paper emerged from set-theoretic investigations into cardinal characteristics of the continuum in \cite{BBBHHL} and \cite{BBH}. In the course of these investigations, the question arose: \emph{How small can a collection $\C$ of conditionally convergent series be with the property that every $A \sub \N$ fails to send some member of $\C$ to infinity?} (Specifically, the answer to this question is the so-called ``Galois-Tukey dual'' -- see \cite{Vojtas}, or section 4 of \cite{Blass} -- of the uncountable cardinal $\ss_i$ as defined in \cite{BBH}.) We suspected that the answer to this question should be an uncountable cardinal number. Without too much difficulty, one may prove upper bounds for this number of $\mathrm{non}(\mathcal N)$ and $\mathrm{non}(\mathcal M)$, the smallest size of a non-null and non-meager subset of $\R$, respectively, and $\mathfrak{r}$, the so-called \emph{reaping number}, another uncountable cardinal. (Proofs of these upper bounds can be found, in dual form, in \cite{BBH}). It was a surprise to discover that the correct answer to this question is $4$.

The situation is starkly different when one considers rearrangements rather than subseries. Using ideas related to the Levy-Steinitz theorem \cite{rosen}, one may show (see \cite{BBBHHL}, Section 6) that for any countable collection $\set{\sum_na^i_n}{i \in \w}$ of conditionally convergent series, there is a single permutation $\pi$ of $\N$ such that each of the rearranged series $\sum_na_{\pi(n)}$ sums to $\infty$ or to $-\infty$. In other words, if we change the question from the previous paragraph by replacing the idea of taking subseries by the idea of taking rearrangements of a series, then we really do get an uncountable cardinal number. 

\vspace{3mm}

If we replace ``three" with ``two" in the statement of the main theorem, then it becomes relatively easy to prove. We will go through the proof of this easier result now, because doing so will illuminate some of the difficulties involved in proving the main theorem.
\begin{theorem*}
For any two conditionally convergent series, there is some $A \sub \N$ sending both series to infinity.
\end{theorem*} 
\begin{proof}
Let $\sum_{n \in \N}a_n$ and $\sum_{n \in \N}b_n$ be conditionally convergent series. Partition $\N$ into four sets, depending on where these two series have their positive and non-positive terms:
$$A^{++} = \set{n}{a_n,b_n > 0} \qquad \qquad A^{-+} = \set{n}{a_n \leq 0, b_n > 0}$$
$$A^{+-} = \set{n}{a_n > 0,b_n \leq 0} \qquad \, A^{--} = \set{n}{a_n,b_n \leq 0}. \qquad\!\!\!\!$$
Consider the following eight series:

\begin{center}
\begin{tikzpicture}[xscale=.56,yscale=.45]

\node at (.5,4.06) {\small $\sum_{n \in A^{++}}a_n$,};
\node at (.5,2.94) {\small $\sum_{n \in A^{++}}b_n$};

\node at (.5,1.06) {\small $\sum_{n \in A^{+-}}a_n$,};
\node at (.5,-.06) {\small $\sum_{n \in A^{+-}}b_n$};

\node at (7.5,4.06) {\small $\sum_{n \in A^{-+}}a_n$,};
\node at (7.5,2.94) {\small $\sum_{n \in A^{-+}}b_n$};

\node at (7.5,1.06) {\small $\sum_{n \in A^{--}}a_n$,};
\node at (7.5,-.06) {\small $\sum_{n \in A^{--}}b_n$};

\node at (.5,5.5) {\footnotesize $a_n > 0$};
\node at (7.5,5.5) {\footnotesize $a_n \leq 0$};
\node at (-4,3.5) {\footnotesize $b_n > 0$};
\node at (-4,.5) {\footnotesize $b_n \leq 0$};

\draw[thin] (4,-1) -- (4,6);
\draw[thin] (11,-1) -- (11,6);
\draw[thin] (-3,-1) -- (-3,6);
\draw[thin] (-5,-1) -- (-5,5);
\draw[thin] (-5,2) -- (11,2);
\draw[thin] (-5,5) -- (11,5);
\draw[thin] (-5,-1) -- (11,-1);
\draw[thin] (-3,6) -- (11,6);

\end{tikzpicture}
\end{center}
\vspace{-2mm}

\begin{claim}
At least one of $\sum_{n \in A^{++}}a_n$ and $\sum_{n \in A^{+-}}a_n$ must sum to $\infty$, and if the other one does not then it must be absolutely convergent.
\end{claim}
\begin{proof}[Proof of claim:]
All the terms of $\sum_{n \in A^{++}}a_n$ and $\sum_{n \in A^{+-}}a_n$ are positive. This implies that each of these series must either sum to $\infty$, or else be absolutely convergent. But they cannot both be absolutely convergent: $\sum_{n \in A^{++} \cup A^{+-}}a_n = \infty$, because this is just the sum over all positive $a_n$, so splitting this series into the two subseries $\sum_{n \in A^{++}}a_n$ and $\sum_{n \in A^{+-}}a_n$ cannot result in two absolutely convergent series.
\end{proof}

In other words, this claim says that one of the $a_n$ subseries from the first column must sum to $\infty$, and if the other one doesn't then it must be absolutely convergent.  
A similar statement holds concernlng the $a_n$ subseries from the second column: at least one must sum to $-\infty$, and if the other one doesn't then it must be absolutely convergent. Likewise, a similar statement holds concernlng the two $b_n$ subseries from the first row, and to the two $b_n$ subseries from the second row.

Let us now try to find a set that sends both series to infinity. If one of the four sets $A^{++}, A^{+-}, A^{-+}, A^{--}$ works, then we are done. So let us suppose that none of these four sets sends both series to infinity.

If $A^{++}$ does not send both our series to infinity, then (by our claim) it must be because one or both of $\sum_{n \in A^{++}}a_n$ and $\sum_{n \in A^{++}}b_n$ are absolutely convergent. Let us suppose for now that both are absolutely convergent. By the claim above (and its variations involving the $b_n$), this implies 
\begin{itemize}
\item[$\circ$] $\sum_{n \in A^{+-}}a_n = \infty$ and $\sum_{n \in A^{-+}}b_n = \infty$.
\end{itemize}
As neither $A^{+-}$ nor $A^{-+}$ sends both series to infinity, this implies
\begin{itemize}
\item[$\circ$] $\sum_{n \in A^{+-}}b_n$ and $\sum_{n \in A^{-+}}a_n$ are absolutely convergent.
\end{itemize}
But, again using the claim above, this implies
\begin{itemize}
\item[$\circ$] $\sum_{n \in A^{--}}b_n = -\infty$ and $\sum_{n \in A^{--}}a_n = -\infty$,
\end{itemize}
contrary to our assumption that $A^{--}$ does not send both series to infinity.

This shows that exactly one of $\sum_{n \in A^{++}}a_n$ and $\sum_{n \in A^{++}}b_n$ is absolutely convergent, and the other sums to $\infty$. Using the claim above to reason as in the previous paragraph, this leads us to two symmetric possibilities:

\begin{center}
\begin{tikzpicture}[xscale=.6,yscale=.56]

\node at (.5,4.05) {\small $\sum_{n \in A^{++}}a_n = \infty$,};
\node at (.5,2.95) {\small $\sum_{n \in A^{++}}b_n$ abs. conv.};

\node at (.5,1.05) {\small $\sum_{n \in A^{+-}}a_n$ abs. conv.,};
\node at (.5,-.05) {\small $\sum_{n \in A^{+-}}b_n = -\infty$};

\node at (7.5,4.05) {\small $\sum_{n \in A^{-+}}a_n$ abs. conv.,};
\node at (7.5,2.95) {\small $\sum_{n \in A^{-+}}b_n = \infty$};

\node at (7.5,1.05) {\small $\sum_{n \in A^{--}}a_n = -\infty$,};
\node at (7.5,-.05) {\small $\sum_{n \in A^{--}}b_n$ abs. conv.};

\node at (.5,5.5) {\footnotesize $a_n > 0$};
\node at (7.5,5.5) {\footnotesize $a_n \leq 0$};
\node at (-4,3.5) {\footnotesize $b_n > 0$};
\node at (-4,.5) {\footnotesize $b_n \leq 0$};

\draw[thin] (4,-1) -- (4,6);
\draw[thin] (11,-1) -- (11,6);
\draw[thin] (-3,-1) -- (-3,6);
\draw[thin] (-5,-1) -- (-5,5);
\draw[thin] (-5,2) -- (11,2);
\draw[thin] (-5,5) -- (11,5);
\draw[thin] (-5,-1) -- (11,-1);
\draw[thin] (-3,6) -- (11,6);

\end{tikzpicture}
\end{center}

\begin{center}
\begin{tikzpicture}[xscale=.6,yscale=.56]

\node at (.5,4.05) {\small $\sum_{n \in A^{++}}a_n$ abs. conv.,};
\node at (.5,2.95) {\small $\sum_{n \in A^{++}}b_n = \infty$};

\node at (.5,1.05) {\small $\sum_{n \in A^{+-}}a_n = \infty$,};
\node at (.5,-.05) {\small $\sum_{n \in A^{+-}}b_n$ abs. conv.};

\node at (7.5,4.05) {\small $\sum_{n \in A^{-+}}a_n = -\infty$,};
\node at (7.5,2.95) {\small $\sum_{n \in A^{-+}}b_n$ abs. conv.};

\node at (7.5,1.05) {\small $\sum_{n \in A^{--}}a_n$ abs. conv.,};
\node at (7.5,-.05) {\small $\sum_{n \in A^{--}}b_n = -\infty$};

\node at (.5,5.5) {\footnotesize $a_n > 0$};
\node at (7.5,5.5) {\footnotesize $a_n \leq 0$};
\node at (-4,3.5) {\footnotesize $b_n > 0$};
\node at (-4,.5) {\footnotesize $b_n \leq 0$};

\draw[thin] (4,-1) -- (4,6);
\draw[thin] (11,-1) -- (11,6);
\draw[thin] (-3,-1) -- (-3,6);
\draw[thin] (-5,-1) -- (-5,5);
\draw[thin] (-5,2) -- (11,2);
\draw[thin] (-5,5) -- (11,5);
\draw[thin] (-5,-1) -- (11,-1);
\draw[thin] (-3,6) -- (11,6);

\end{tikzpicture}
\end{center}

\noindent In either case, it is easy to see that $A = A^{++} \cup A^{+-}$ sends both our series to infinity.
\end{proof}

This proof, though a little space-consuming, is driven by a very simple idea. The strategy of the proof can be summed up in a sentence: First partition $\N$ into $4$ sets according to where our two series are positive/non-positive, then show that if one of these $4$ sets does not already send both our series to infinity, then the union of two of them does.

This ``partition-and-union'' strategy does not work for three series. Given three conditionally convergent series, one may of course partition $\N$ into $8$ sets according to where each of the three series have their positive and non-positive terms. By combining these sets in various ways, we obtain up to $255$ nonempty subsets of $\N$. If the partition-and-union strategy from the previous paragraph is to work, then one of these $255$ sets should do the job of sending all three series to infinity simultaneously. But this is not always so.
Consider the following three conditionally convergent series:

\vspace{-3mm}
\renewcommand*{\arraystretch}{2}
$$\begin{matrix}

1 & \textstyle - \ \frac{1}{2} & \textstyle + \ \frac{1}{3} & \textstyle - \ \frac{1}{4} & \textstyle + \ \frac{1}{5} & \textstyle - \ \frac{1}{6} & \textstyle + \ \frac{1}{7} & \textstyle - \ \frac{1}{8} & \textstyle + \ \frac{1}{9} & \textstyle - \ \frac{1}{10} \ \  \dots \\

1 & \textstyle + \ 0 & \textstyle - \ \frac{1}{3} & \textstyle + \ 0 & \textstyle + \ \frac{1}{5} & \textstyle + \ 0 & \textstyle - \ \frac{1}{7} & \textstyle + \ 0 & \textstyle + \ \frac{1}{9} & \textstyle + \ \, 0 \ \ \ \dots \\

0 & \textstyle + \ \frac{1}{2} & \textstyle + \ 0 & \textstyle - \ \frac{1}{4} & \textstyle + \ 0 & \textstyle + \ \frac{1}{6} & \textstyle + \ 0 & \textstyle - \ \frac{1}{8} & \textstyle + \ 0 & \textstyle + \ \frac{1}{10} \ \  \dots

\end{matrix}$$
\vspace{.25mm}

If we partition $\N$ into sets according to where the above three series are positive or not, we obtain only four sets, not eight, namely the mod-$4$ equivalence classes. This partition generates only $15$ nonempty subsets of $\N$, rather than $255$, and one can check by hand that none of these $15$ sets succeeds in sending all three series to infinity simultaneously.

Part of what made the proof so simple for two series is that we never had to deal with a subseries $\sum_{n \in A}a_n$ having ``many'' terms of both signs. Again, for three series this is no longer the case. 
In our example above, it is not too difficult to see that if $A \sub \N$ sends sends the second series to infinity, then $A$ must contain enough odd numbers to make the sum of the positive terms from the first series infinite:
$$\textstyle \sum_{\substack{n \in A, \\ n \text{ odd}}} \frac{(-1)^{n+1}}{n} \,=\, \infty.$$
Similarly, if $A$ sends the third series to infinity then $A$ must contain enough even numbers so that
$$\textstyle \sum_{\substack{n \in A, \\ n \text{ even}}} \frac{(-1)^{n+1}}{n} \,=\, -\infty.$$
Thus, in order to send all three series to infinity, $A$ must contain ``many'' positive and negative terms from the first series. 

This complicates the proof of our main theorem. It is not always possible to find some $A \sub \N$ on which each of our series has everywhere (or almost everywhere) the same sign. Instead, there is a balancing act that must be achieved, and the set $A$ must, for at least one of our three series, contain ``many'' positive and negative terms together, but with one sign consistently pulling harder than the other in their infinite game of tug-of-war.

\section{Sending three series to infinity}\label{sec:main}

In this section we prove the main theorem, beginning with a sequence of definitions and lemmas. 

\begin{lemma}\label{lem:partition}
Let $\sum_{n \in \N}a^1_n$, $\sum_{n \in \N}a^2_n$, and $\sum_{n \in \N}a^3_n$ be conditionally convergent series. Then there is a partition $\P$ of $\N$ (into at most $8$ sets) such that for each $i \in \{1,2,3\}$, 
\begin{itemize}
\item[\footnotesize \raisebox{.15mm}{$\bullet$}] there is some $A \in \P$ such that $\sum_{i \in A}a^i_n = \infty$,
\item[\footnotesize \raisebox{.15mm}{$\bullet$}] there is some $A \in \P$ such that $\sum_{i \in A}a^i_n = -\infty$, and
\item[\footnotesize \raisebox{.15mm}{$\bullet$}] for every $A \in \P$, the terms of $\sum_{i \in A}a^i_n$ all have the same sign.
\end{itemize}
\end{lemma}
\begin{proof}
This is exactly as in the previous section: just partition $\N$ into $8$ sets according to where the terms of each series are positive/non-positive, and argue as in the Claim from the previous proof.
\end{proof}

\begin{definition}
Suppose $\sum_{n \in \N}a_n$ is a conditionally convergent series, and let $A \sub \N$. We say that $A$ is \emph{tame with respect to $\sum_{n \in \N}a_n$} if at most one of the subseries 
$$\textstyle \sum_{n \in A,\, a_n > 0}\,a_n \qquad \text{and} \qquad \textstyle \sum_{n \in A,\, a_n \leq 0}\,a_n$$
fails to be absolutely convergent. When the series is clear from context, we may say simply that $A$ is tame.
Given three conditionally convergent series $\sum_{n \in \N}a^1_n$, $\sum_{n \in \N}a^2_n$, and $\sum_{n \in \N}a^3_n$, we say that $A \sub \N$ is \emph{tame} if it is tame with respect to each of the three series.
\end{definition}

In other words, $A \sub \N$ is tame with respect to $\sum_{n \in \N}a_n$ provided that it does not contain ``many'' positive and negative terms at the same time. Notice that tame sets are precisely the kind that might arise when applying the partition-and-union strategy from the previous section. The following three lemmas are easy exercises related to tameness and we omit the proofs, but we state the lemmas explicitly because they play an important part in what follows:

\begin{lemma}\label{lem:full}
Let $\sum_{n \in \N}a_n$ be a conditionally convergent series and $A \sub \N$. If all the terms of $\sum_{n \in A}a_n$ have the same sign, then $A$ is tame.
\end{lemma}

\begin{lemma}\label{lem:subset}
Every subset of a tame set is tame.
\end{lemma}

In the statement of the following two lemmas, we use the usual conventions for extended real numbers, so that the sum $\sum_{n \in A}a_n + \sum_{n \in B}a_n$ is well-defined unless either $\sum_{n \in A}a_n = \infty$ and $\sum_{n \in B}a_n = -\infty$, or
$\sum_{n \in A}a_n = -\infty$ and $\sum_{n \in B}a_n = \infty$. 

\begin{lemma}\label{lem:unionclosed}
Let $\sum_{n \in \N}a_n$ be a conditionally convergent series, and let $A,B \sub \N$ be tame. If $\sum_{n \in A}a_n + \sum_{n \in B}a_n$ is well-defined, then $A \cup B$ is tame and
$\sum_{n \in A}a_n + \sum_{n \in B}a_n = \sum_{n \in A \cup B}a_n.$
\end{lemma}

We now turn to the question of what combinatorial obstacles might arise in attempting to apply the partition-and-union strategy to three series instead of two. 

\begin{definition}
Let $\mathrm{Fn}(3,2)$ denote the set of all functions from a nonempty subset of $\{1,2,3\}$ into the $2$-element set $\{p,n\}$. 
\begin{itemize}
\item[\footnotesize \raisebox{.15mm}{$\bullet$}] We say that $\F \sub \mathrm{Fn}(3,2)$ is \emph{full} provided that for every $x \in \{1,2,3\}$ and $y \in \{p,n\}$, there is some $f \in \F$ with $f(x) = y$.
\item[\footnotesize \raisebox{.15mm}{$\bullet$}] We say that $\F \sub \mathrm{Fn}(3,2)$ is \emph{union-closed} provided that for every $f,g \in \F$, if $f$ and $g$ are compatible (which means that they take the same value on any member of $\mathrm{dom}(f) \cap \mathrm{dom}(g)$) then $f \cup g \in \F$.
\end{itemize}
\end{definition}

The idea here is that a member of $\mathrm{Fn}(3,2)$ can represent the effect of a tame set on our three series. More precisely:

\begin{definition}
Let $\sum_{n \in \N}a^1_n$, $\sum_{n \in \N}a^2_n$, and $\sum_{n \in \N}a^3_n$ be conditionally convergent series. Given $A \sub \N$, let $\phi_A$ denote the function from some subset of $\{1,2,3\}$ to $\{p,n\}$ defined as follows:
$$\phi_A(i) = 
\begin{cases}
p &\text{ if } \sum_{n \in A}a^i_n = \infty, \\
n &\text{ if } \sum_{n \in A}a^i_n = -\infty,
\end{cases}$$
and otherwise $i \notin \mathrm{dom}(\phi_A)$.
\end{definition}

\begin{lemma}\label{lem:tameness}
Let $\sum_{n \in \N}a^1_n$, $\sum_{n \in \N}a^2_n$, and $\sum_{n \in \N}a^3_n$ be conditionally convergent series. Then
$$\F = \textstyle \set{\phi_A}{A \sub \N \text{ is tame and }\phi_A \text{ is not the empty function}}$$
is a full, union-closed subset of $\mathrm{Fn}(3,2)$.
\end{lemma}
\begin{proof}
It is clear from our definitions that $\F \sub \mathrm{Fn}(3,2)$. 
Taken together, Lemmas~\ref{lem:partition} and \ref{lem:full} imply that $\F$ is full. Similarly, Lemma~\ref{lem:unionclosed} implies that $\F$ is union-closed.
\end{proof}

In what follows, we shall sometimes represent functions as sets of ordered pairs, and sometimes we shall represent functions and sets of functions as pictures. For example, 

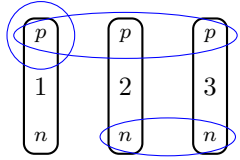
\vspace{-1mm}
\begin{center}
\begin{tikzpicture}[xscale=.9,yscale=.9]

\draw[rounded corners,thick] (0,0) rectangle (.5,2) {};
\draw[rounded corners,thick] (1.25,0) rectangle (1.75,2) {};
\draw[rounded corners,thick] (2.5,0) rectangle (3,2) {};
\node at (.25,.25) {\scriptsize $n$};
\node at (.25,1.75) {\scriptsize $p$};
\node at (1.5,.25) {\scriptsize $n$};
\node at (1.5,1.75) {\scriptsize $p$};
\node at (2.75,.25) {\scriptsize $n$};
\node at (2.75,1.75) {\scriptsize $p$};
\node at (.25,1) {\small $1$};
\node at (1.5,1) {\small $2$};
\node at (2.75,1) {\small $3$};

\draw[blue] (.25,1.75) circle (.5cm);
\draw[blue] (1.5,1.75) ellipse (1.65cm and .34cm);
\draw[blue] (2.125,.25) ellipse (1cm and .28cm);

\end{tikzpicture}
\end{center}

\noindent represents the subset of $\mathrm{Fn}(3,2)$ containing the three functions $\{(1,p)\}$, $\{(1,p),(2,p),(3,p)\}$, and $\{(2,n),(3,n)\}$.

\begin{definition}
Two sets $\F,\G \sub \mathrm{Fn}(3,2)$ are \emph{equivalent} if there is a permutation $\pi$ of $\{1,2,3\}$ and permutations $\tau_1$, $\tau_2$, and $\tau_3$ of $\{p,n\}$ such that
$$g \in \G  \quad \Leftrightarrow \quad g(x) = \tau_x \circ f \circ \pi (x) \ \text{ for some } f \in \F.$$
\end{definition}

The intuition here is that two subsets of $\mathrm{Fn}(3,2)$ are equivalent if they have the same ``shape.'' For example, consider the two sets of functions pictured below:

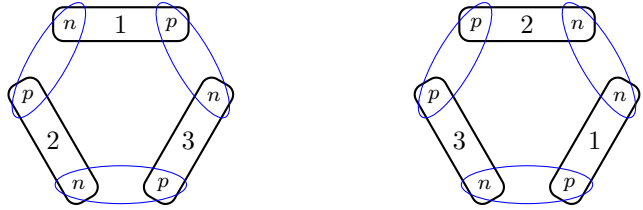
\begin{center}
\begin{tikzpicture}[xscale=.9,yscale=.9]

\draw[rounded corners,thick] (-1,3.206) rectangle (1,3.706) {};
\node at (-.75,3.45) {\scriptsize $n$};
\node at (.75,3.45) {\scriptsize $p$};
\draw[rotate=30,rounded corners,thick] (-.25,1) rectangle (.25,3) {};
\node at (-1.375,2.382) {\scriptsize $p$};
\node at (1.375,2.382) {\scriptsize $n$};
\node at (-.625,1.083) {\scriptsize $n$};
\node at (.625,1.083) {\scriptsize $p$};
\draw[rotate=-30,rounded corners,thick] (-.25,1) rectangle (.25,3) {};
\node at (0,3.456) {\small $1$};
\node at (-1,1.733) {\small $2$};
\node at (1,1.733) {\small $3$};

\draw[blue] (0,1.083) ellipse (.97cm and .28cm);
\draw[rotate=-60,blue] (-2,2.382) ellipse (.97cm and .28cm);
\draw[rotate=60,blue] (2,2.382) ellipse (.97cm and .28cm);

\draw[rounded corners,thick] (5,3.206) rectangle (7,3.706) {};
\node at (5.25,3.45) {\scriptsize $p$};
\node at (6.75,3.45) {\scriptsize $n$};
\draw[rotate=30,rounded corners,thick] (4.946,-2) rectangle (5.446,0) {};
\node at (4.625,2.382) {\scriptsize $p$};
\node at (7.375,2.382) {\scriptsize $n$};
\node at (5.375,1.083) {\scriptsize $n$};
\node at (6.625,1.083) {\scriptsize $p$};
\draw[rotate=-30,rounded corners,thick] (4.946,4) rectangle (5.446,6) {};
\node at (6,3.456) {\small $2$};
\node at (5,1.733) {\small $3$};
\node at (7,1.733) {\small $1$};

\draw[blue] (6,1.083) ellipse (.97cm and .28cm);
\draw[rotate=-60,blue] (1,7.578) ellipse (.97cm and .28cm);
\draw[rotate=60,blue] (5,-2.814) ellipse (.97cm and .28cm);

\end{tikzpicture}
\end{center}

\noindent These two pictures represent different subsets of $\mathrm{Fn}(3,2)$. But these two sets of functions are equivalent (in the sense defined above), precisely because their pictures can be changed from one to the other simply by relabelling. When we have the need to represent a family of functions pictorially in what follows, we shall sometimes draw pictures with (some) labels missing, as a reminder that we really only care about sets of functions up to equivalence.

\begin{lemma}\label{lem:twotypes}
Up to equivalence, there are exactly four full, union-closed subsets of $\mathrm{Fn}(3,2)$ that do not include a function with domain $\{1,2,3\}$. These four sets of functions come in two types:

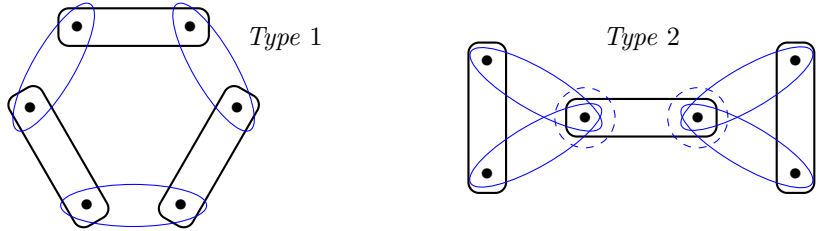
\begin{center}
\begin{tikzpicture}

\node at (2,3.3) {\small Type $1$};
\draw[rounded corners,thick] (-1,3.206) rectangle (1,3.706) {};
\node at (-.75,3.45) {\small $\bullet$};
\node at (.75,3.45) {\small $\bullet$};
\draw[rotate=30,rounded corners,thick] (-.25,1) rectangle (.25,3) {};
\node at (-1.375,2.382) {\small $\bullet$};
\node at (1.375,2.382) {\small $\bullet$};
\node at (-.625,1.083) {\small $\bullet$};
\node at (.625,1.083) {\small $\bullet$};
\draw[rotate=-30,rounded corners,thick] (-.25,1) rectangle (.25,3) {};

\draw[blue] (0,1.083) ellipse (.97cm and .28cm);
\draw[rotate=-60,blue] (-2,2.382) ellipse (.97cm and .28cm);
\draw[rotate=60,blue] (2,2.382) ellipse (.97cm and .28cm);

\node at (6.75,3.3) {\small Type $2$};
\draw[rounded corners,thick] (5.75,2) rectangle (7.75,2.5) {};
\node at (6,2.25) {\small $\bullet$};
\node at (7.5,2.25) {\small $\bullet$};
\draw[rounded corners,thick] (4.451,1.25) rectangle (4.951,3.25) {};
\node at (4.701,1.5) {\small $\bullet$};
\node at (4.701,3) {\small $\bullet$};
\draw[rounded corners,thick] (8.549,1.25) rectangle (9.049,3.25) {};
\node at (8.799,1.5) {\small $\bullet$};
\node at (8.799,3) {\small $\bullet$};

\draw[rotate=30,blue] (5.57,-1.042) ellipse (1cm and .28cm);
\draw[rotate=-30,blue] (3.318,4.95) ellipse (1cm and .28cm);
\draw[rotate=30,blue] (8.379,-1.792) ellipse (1cm and .28cm);
\draw[rotate=-30,blue] (6.123,5.7) ellipse (1cm and .28cm);
\draw[blue,dashed] (6,2.25) circle (.4cm);
\draw[blue,dashed] (7.5,2.25) circle (.4cm);

\end{tikzpicture}
\end{center}

\noindent Type $1$ consists of just one set of functions (up to equivalence), and Type $2$ consists of three: each Type 2 set contains all four of the functions indicated by ellipses, but non-equivalent sets are obtained by including zero, one, or both of the functions indicated by dashed circles.
\end{lemma}
\begin{proof}
Suppose $\F$ is a full, union-closed subset of $\mathrm{Fn}(3,2)$ that does not contain a function with domain $\{1,2,3\}$.

If every $f \in \F$ were to have $\card{\mathrm{dom}(f)} = 1$, then the fullness of $\F$ would imply that the functions $\{(1,p)\}$, $\{(2,p)\}$, and $\{(3,p)\}$ are all in $\F$, but then union-closedness of $\F$ would imply $\{(1,p),(2,p),(3,p)\} \in \F$.







Thus $\F$ must contain a function $f$ with $\card{\mathrm{dom}(f)} = 2$. Let $z$ denote the member of $\{1,2,3\}$ not in the domain of $f$. By fullness, some other $g \in \F$ must have $z \in \mathrm{dom}(g)$. But then $\mathrm{dom}(f) \cup \mathrm{dom}(g) = \{1,2,3\}$ so, because $\F$ is union-closed and contains no functions with domain $\{1,2,3\}$, $f$ and $g$ must be incompatible. Because $z \in \mathrm{dom}(g)$ and $|\mathrm{dom}(g)| \leq 2$, this implies $|\mathrm{dom}(g)| = 2$, $|\mathrm{dom}(f) \cap \mathrm{dom}(g)| = 1$, and $f$ and $g$ take different values on the unique element of $\mathrm{dom}(f) \cap \mathrm{dom}(g)$.

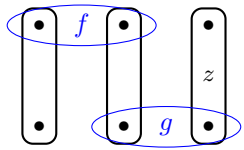
\begin{center}
\begin{tikzpicture}[xscale=.9,yscale=.9]

\draw[rounded corners,thick] (0,0) rectangle (.5,2) {};
\draw[rounded corners,thick] (1.25,0) rectangle (1.75,2) {};
\draw[rounded corners,thick] (2.5,0) rectangle (3,2) {};
\node at (.25,.25) {\footnotesize $\bullet$};
\node at (.25,1.75) {\footnotesize $\bullet$};
\node at (1.5,.25) {\footnotesize $\bullet$};
\node at (1.5,1.75) {\footnotesize $\bullet$};
\node at (2.75,.25) {\footnotesize $\bullet$};
\node at (2.75,1.75) {\footnotesize $\bullet$};
\node at (2.75,1) {\footnotesize $z$};
\node[blue] at (.875,1.75) {\small $f$};
\node[blue] at (2.125,.25) {\small $g$};

\draw[blue] (.875,1.75) ellipse (1.1cm and .3cm);
\draw[blue] (2.125,.25) ellipse (1.1cm and .3cm);

\end{tikzpicture}
\end{center}

Let us denote by $y$ the point of $\mathrm{dom}(f) \cap \mathrm{dom}(g)$, and by $x$ the remaining member of $\mathrm{dom}(f)$. 
By the fullness of $\F$, there must be some $h \in \F$ with $h(x) \neq f(x)$. But then observe that $h$ and $g$ must be incompatible, because otherwise union-closedness would give us $g \cup h \in \F$ with $|\mathrm{dom}(g \cup h)| = 3$. There are two possibilities:

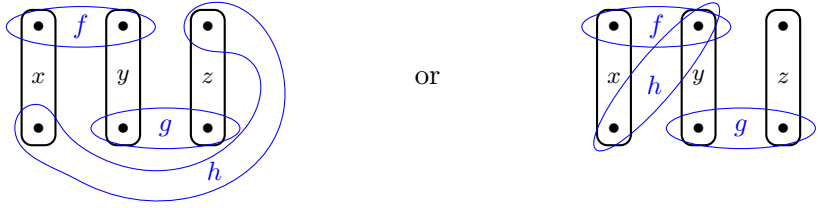
\begin{center}
\begin{tikzpicture}[xscale=.9,yscale=.9]

\begin{scope}[shift={(15.5,0)}]
\draw[rounded corners,thick] (0,0) rectangle (.5,2) {};
\draw[rounded corners,thick] (1.25,0) rectangle (1.75,2) {};
\draw[rounded corners,thick] (2.5,0) rectangle (3,2) {};
\node at (.25,.25) {\footnotesize $\bullet$};
\node at (.25,1.75) {\footnotesize $\bullet$};
\node at (1.5,.25) {\footnotesize $\bullet$};
\node at (1.5,1.75) {\footnotesize $\bullet$};
\node at (2.75,.25) {\footnotesize $\bullet$};
\node at (2.75,1.75) {\footnotesize $\bullet$};
\node at (.25,1) {\footnotesize $x$};
\node at (1.5,1) {\footnotesize $y$};
\node at (2.75,1) {\footnotesize $z$};
\node[blue] at (.875,1.75) {\small $f$};
\node[blue] at (2.125,.25) {\small $g$};
\node[blue] at (.85,.9) {\small $h$};

\draw[blue] (.875,1.75) ellipse (1.1cm and .3cm);
\draw[blue] (2.125,.25) ellipse (1.1cm and .3cm);
\draw[rotate=50.2,blue] (1.34,-.03) ellipse (1.4cm and .32cm);
\end{scope}

\node at (13,1) {or};

\draw[rounded corners,thick] (7,0) rectangle (7.5,2) {};
\draw[rounded corners,thick] (8.25,0) rectangle (8.75,2) {};
\draw[rounded corners,thick] (9.5,0) rectangle (10,2) {};
\node at (7.25,.25) {\footnotesize $\bullet$};
\node at (7.25,1.75) {\footnotesize $\bullet$};
\node at (8.5,.25) {\footnotesize $\bullet$};
\node at (8.5,1.75) {\footnotesize $\bullet$};
\node at (9.75,.25) {\footnotesize $\bullet$};
\node at (9.75,1.75) {\footnotesize $\bullet$};
\node at (7.25,1) {\footnotesize $x$};
\node at (8.5,1) {\footnotesize $y$};
\node at (9.75,1) {\footnotesize $z$};
\node[blue] at (7.875,1.75) {\small $f$};
\node[blue] at (9.125,.25) {\small $g$};
\node[blue] at (9.85,-.38) {\small $h$};

\draw[blue] (7.875,1.75) ellipse (1.1cm and .3cm);
\draw[blue] (9.125,.25) ellipse (1.1cm and .3cm);
\path[draw,blue,use Hobby shortcut,closed=true]
(6.9,.25) .. (7.25,.6) .. (7.6,.25) .. (8,-.1) .. (9.75,-.2) .. (10.3,.25) .. (10.3,1.3) .. (9.75,1.4) .. (9.4,1.75) .. (9.75,2.1) .. (10.8,1.3) .. (10.8,.25) .. (9.75,-.7) .. (8,-.6) .. (7.6,-.4);

\end{tikzpicture}
\end{center}
\vspace{-3mm}

In the first case (the picture on the left), we see that we have found the Type 1 solution described in the statement of the theorem. We claim that in this case $\F = \{f,g,h\}$ (i.e., $\F$ contains no functions other than those already shown). To see this, suppose $e \in \F$. If $\card{\mathrm{dom}(e)} = 1$, then one of $f$, $g$, or $h$ has domain $\{1,2,3\} \setminus \mathrm{dom}(e)$, and union-closedness would give us a function with domain $\{1,2,3\}$. Hence $|\mathrm{dom}(e)| = 2$. If $\mathrm{dom}(e) = \mathrm{dom}(f) = \{x,y\}$, then $e(x) = f(x)$ since otherwise $e$ and $h$ would be compatible and $\card{\mathrm{dom}(e \cup h)} = 3$; likewise, $e(y) = f(y)$ since otherwise $e$ and $g$ would be compatible and $\card{\mathrm{dom}(e \cup g)} = 3$. Hence $\mathrm{dom}(e) = \mathrm{dom}(f)$ implies $e = f$. Similarly, $\mathrm{dom}(e) = \mathrm{dom}(g)$ implies $e = g$ and $\mathrm{dom}(e) = \mathrm{dom}(h)$ implies $e = h$. Thus $\F = \{f,g,h\}$, as claimed.

In the second case (the picture on the right), $\F$ must contain at least one more function (because $\F$ is full). In particular, there is some $e \in \F$ with $z \in \mathrm{dom}(e)$ and $e(z) \neq g(z)$. Now, $e$ must satisfy $|\mathrm{dom}(e)| \leq 2$, but it also must be incompatible with both of $f$ and $h$ (because otherwise union-closedness would give us a function in $\F$ with size-$3$ domain). The only way to accomplish both of these is to have $y \in \mathrm{dom}(e)$ and $e(y) = g(y)$.

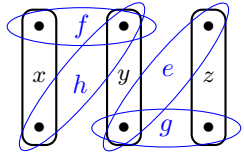
\begin{center}
\begin{tikzpicture}[xscale=.9,yscale=.9]

\draw[rounded corners,thick] (0,0) rectangle (.5,2) {};
\draw[rounded corners,thick] (1.25,0) rectangle (1.75,2) {};
\draw[rounded corners,thick] (2.5,0) rectangle (3,2) {};
\node at (.25,.25) {\footnotesize $\bullet$};
\node at (.25,1.75) {\footnotesize $\bullet$};
\node at (1.5,.25) {\footnotesize $\bullet$};
\node at (1.5,1.75) {\footnotesize $\bullet$};
\node at (2.75,.25) {\footnotesize $\bullet$};
\node at (2.75,1.75) {\footnotesize $\bullet$};
\node at (.25,1) {\footnotesize $x$};
\node at (1.5,1) {\footnotesize $y$};
\node at (2.75,1) {\footnotesize $z$};
\node[blue] at (.875,1.75) {\small $f$};
\node[blue] at (2.125,.25) {\small $g$};
\node[blue] at (.85,.9) {\small $h$};
\node[blue] at (2.15,1.1) {\small $e$};

\draw[blue] (.875,1.75) ellipse (1.1cm and .28cm);
\draw[blue] (2.125,.25) ellipse (1.1cm and .28cm);
\draw[rotate=50.2,blue] (1.34,-.03) ellipse (1.4cm and .3cm);
\draw[rotate=50.2,blue] (2.13,-.985) ellipse (1.4cm and .3cm);

\end{tikzpicture}
\end{center}

This is the smallest of the Type 2 sets from the statement of the lemma.
Suppose that $d \in \F$. Having $\mathrm{dom}(d) = \{x,z\}$ is impossible, because then $d$ would be compatible with one of $f$ or $h$, and their union would give us a function in $\F$ with size-$3$ domain. Thus $y \in \mathrm{dom}(d)$. If $\mathrm{dom}(d) = \{x,y\}$, then $d(y) = f(y) = h(y)$, since otherwise $d$ and $g$ are compatible, and $|\mathrm{dom}(d \cup g)| = 3$; but then, depending on the value of $d(x)$, either $d = f$ or $d = h$. Similarly, if $\mathrm{dom}(d) = \{y,z\}$ then either $d = e$ or $e = g$. Thus $y \in \mathrm{dom}(d)$, and if $|\mathrm{dom}(d)| = 2$ implies $d \in \{e,f,g,h\}$. Put another way, if $d \in \F \setminus \{e,f,g,h\}$, then $\mathrm{dom}(d) = \{y\}$. This gives us exactly the three Type 2 solutions described in the statement of the lemma.
\end{proof}

We are nearly ready to prove the main theorem, but first we need three more lemmas. The first is a corollary of Theorem 3.47 in \cite{Rudin}, so we do not include a proof here. The other two lemmas capture the analysis necessary to perform the ``balancing act'' described in the introduction.

\begin{lemma}\label{lem:disjoint}
Let $\sum_{n \in \N}a_n$ be a conditionally convergent series, and let $A,B \sub \N$. If $A$ and $B$ are disjoint, then 
$$\textstyle \sum_{n \in A \cup B}a_n = \sum_{n \in A}a_n + \sum_{n \in B}a_n$$ 
whenever the expression on the right is well-defined. Consequently,
$$\textstyle \sum_{n \in C \setminus D}a_n = \sum_{n \in C \cup D}a_n - \sum_{n \in D}a_n$$
for any $C,D \sub \N$, whenever the expression on the right is well-defined. 
\end{lemma}

Compare the first part of Lemma~\ref{lem:disjoint} with Lemma~\ref{lem:unionclosed}; the conclusions are the same, but we have replaced the assumption that $A$ and $B$ are tame with the assumption that they are disjoint.

\begin{lemma}\label{lem:balance1}
Let $ \sum_{n \in \N} a_n$ be a conditionally convergent series, and let $A \sub \N$ be tame. If $\sum_{n \in A} \,a_n = \infty$, then
there is a set $B \sub A$ such that
$$\textstyle \sum_{n \in B} \,a_n = \infty \quad \text{and} \quad
 \sum_{n \in A \setminus B} \,a_n = \infty.$$
\end{lemma}
\begin{proof}
Using recursion, we may define an increasing sequence $k_0,k_1,k_2,\dots$ of integers as follows. To begin, let $k_0 = 0$. Given $k_0, k_1, k_2, \dots, k_{m-1}$, define $k_m$ to be the least natural number such that $\sum_{n \in A \cap (k_{m-1},k_m]}|a_n| > 1$. Such a $k_m$ must exist because $\sum_{n \in A}|a_n| = \infty$. 

Let $B = \bigcup_{m \in \N}(k_{2m-1},k_{2m}]$.
Because $A$ is tame and $\sum_{n \in A}a_n = \infty$, there are only two options: either $\sum_{n \in B}a_n = \infty$ or $\sum_{n \in B}a_n$ converges absolutely. 
The latter option is ruled out by our construction of the $k_m$, so $\sum_{n \in B}a_n = \infty$. A similar argument shows that $\sum_{n \in A \setminus B}a_n = \infty$.
\end{proof}

It is worth noting that Lemma~\ref{lem:balance1} remains true even if we do not require $A$ to be tame (but a different, more complicated proof is required for this).

\begin{lemma}\label{lem:balance2}
Let $ \sum_{n \in \N} a_n$ be a conditionally convergent series and let $A \sub \N$. Suppose both $A$ and $\N \setminus A$ are tame, and that
$$\textstyle \sum_{n \in A}a_n = \infty \quad \text{ and } \quad \sum_{n \in \N \setminus A}a_n = -\infty.$$
For any $C \sub A$, there is some $B \sub \N \setminus A$ such that $\sum_{n \in C \cup B}a_n$ converges.
\end{lemma}
\begin{proof}
If $\sum_{n \in C}a_n$ converges, then we may take $B = \0$; and if $\sum_{n \in C \cup (\N \setminus A)}a_n$ converges, then we may take $B = \N \setminus A$. Thus we may (and do) assume, for the remainder of the proof, that neither $\sum_{n \in C}a_n$ nor $\sum_{n \in C \cup (\N \setminus A)}a_n$ converges.

\begin{claim1}
$\sum_{n \in C,\,a_n>0}\,a_n = \infty$ and $\sum_{n \in C,\,a_n \leq 0}\,a_n \text{ converges absolutely.}$
\end{claim1}
\begin{proof}[Proof of claim]
Because $\sum_{n \in A}a_n = \infty$ and $A$ is tame, $\sum_{n \in A,\,a_n \leq 0}a_n$ converges absolutely. As $C \sub A$, $\sum_{n \in C,\,a_n \leq 0}\,a_n \text{ converges absolutely.}$ But $\sum_{n \in C}a_n$ does not converge, so this implies $\sum_{n \in C,\,a_n>0}\,a_n = \infty$.
\end{proof}

\begin{claim2}
$\sum_{n \in C \cup (\N \setminus A)}\,a_n = -\infty$.
\end{claim2}
\begin{proof}[Proof of claim]
By Lemma~\ref{lem:disjoint},
$$\textstyle \sum_{n \in C \cup (\N \setminus A)}\,a_n = \sum_{n \in \N}\,a_n - \sum_{n \in A \setminus C}\,a_n.$$
As $\sum_{n \in \N}a_n$ converges but $\sum_{n \in C \cup (\N \setminus A)}a_n$ does not, this equation implies that $\sum_{n \in A \setminus C}a_n$ does not converge. But $A$ is tame and $\sum_{n \in A} = \infty$, so it follows that $\sum_{n \in A \setminus C}a_n = \infty$. Plugging this back into the equation above, we get $\sum_{n \in C \cup (\N \setminus A)}a_n = -\infty$.
\end{proof}

Define $B \sub \N \setminus A$ recursively according to the following simple rule. If it has already been determined for every $n < j$ whether $n \in B$, then put 
$$\textstyle j \in B \quad \text{ if and only if } \quad
j \notin A \,\text{ and } \, \sum_{n \in C \cup B,\, n < j} \,a_n \,>\, 0.$$

\noindent We claim that $\sum_{n \in C \cup B}a_n$ converges to $0$. Fix $\e > 0$, and for each $j \in \N$ let $S(j)$ denote the finite partial sum $\sum_{n \in C \cup B,\, n < j}\,a_n$. 

\begin{claim3}
$B$ is infinite.
\end{claim3}
\begin{proof}[Proof of claim]
Aiming for a contradiction, suppose $B$ is finite. Because $\sum_{n \in C}a_n = \infty$, this implies $\sum_{n \in C \cup B}a_n = \infty$, which implies that $S(j) > 0$ for all but finitely many $j$. However, by the construction of $B$, if $S(j) > 0$ for all but finitely many $j$ then $B$ is a co-finite subset of $\N \setminus A$. But $\N \setminus A$ is infinite, so this makes $B$ infinite too.
\end{proof}

\begin{claim4}
$S(j) < \e$ for infinitely many $j \in B$.
\end{claim4}
\begin{proof}[Proof of claim]
Aiming for a contradiction, suppose that $S(j) \geq \e$ for all but finitely many $j \in B$ and fix $M \in \N$ greater than all $j \in B$ with $S(j) < \e$. Because $\sum_{n \in C,\,a_n \leq 0}\,a_n$ converges absolutely, all sufficiently large $N$ have the property that if $\ell \geq k \geq N$ then $\sum_{n \in [k,\ell] \cap C} a_n > -\e$. Because $B$ is infinite, we may fix some $N \in B$ with this property. For any $j > N$, let $j_0 = \max\{i \in B : i \leq j\}$ and observe that
$$\textstyle S(j) = S(j_0) + \sum_{n \in (j_0,j] \cap C}a_n > \e + (-\e) = 0$$
which shows that $S(j) > 0$ for all $j > N$. But then, as in the proof of the previous claim, this implies that $B$ is a co-finite subset of $\N \setminus A$. But $\sum_{n \in C \cup (\N \setminus A)}\,a_n = -\infty$ by Claim 2, so if $B$ is a co-finite subset of $\N \setminus A$ then $\sum_{n \in C \cup B}\,a_n = -\infty$. This implies that $S(j) < \e$ for all but finitely many $j$, which is the desired contradiction.
\end{proof}

For all sufficiently large $N \in \N$, and for all $\ell \geq k \geq N$,
\begin{itemize}
\item[$(i)$] $|a_\ell| < \e$, 
\item[$(ii)$] $\sum_{n \in [k,\ell] \cap C} a_n > -\e$,
\item[$(iii)$] $\sum_{n \in [k,\ell] \cap (A \setminus C)} a_n > -\e$,
\item[$(iv)$] $\sum_{n \in [k,\ell]} a_n < \e$.
\end{itemize}
Parts $(i)$ and $(iv)$ hold for all sufficiently large $N$ because $\sum_{n \in \N}a_n$ converges; $(ii)$ holds for all sufficiently large $N$ because $\sum_{n \in C,\,a_n \leq 0}\,a_n$ converges; $(iii)$ holds for all sufficiently large $N$ because $\sum_{n \in A \setminus C,\,a_n \leq 0}\,a_n$ converges.

Using Claims 3 and 4, we may fix some $N \in B$ such that $S(N) < \e$ and such that $N$ satisfies $(i)$ - $(iv)$. 

\begin{claim5}
If $j > N$ then $-2\e < S(j) < 3\e$.
\end{claim5}
\begin{proof}[Proof of claim]
Let $j > N$. To show $S(j) > -2\e$, let ${j_0 = \max\{i \in B : i \leq j\}}$ and observe that
$$\textstyle S(j) \,=\, S(j_0-1) + a_{j_0} + \sum_{n \in (j_0,j] \cap C}a_n.$$
Clearly $j_0 \in B$, so our construction of $B$ implies that $S(j_0-1) > 0$. Also recall that $N \in B$, so that $j_0 \geq N$. By our choice of $N$ (specifically parts $(i)$ and $(ii)$), this implies $a_{j_0} > -\e$ and $\sum_{n \in (j_0,j] \cap C}a_n > -\e$.
Combining these inequalities with the equation above, we get $S(j) > -2\e$.

For the second inequality, if $S(j) < \e$ then there is nothing to prove, so let us suppose that $S(j) \geq \e$.
Let $j_1 = \max\set{i \leq j}{S(i) < \e}$. Recall that $S(N) < \e$, so $j_1$ is well-defined and $j_1 \geq N$.
Now $S(j_1+1) \geq \e$, so by our choice of $N$ (specifically part $(i)$), $S(j_1) \geq 0$, which means $S(i) \geq 0$ for all $j_1 \leq i \leq j$. 
From this fact and from our definition of $B$ it follows that $B \cap [j_1,j] = [j_1,j] \setminus A$. Hence
\begin{align*}
S(j) & = \textstyle S(j_1) + \sum_{n \in (j_1,j] \cap C} a_n + \sum_{n \in (j_1,j] \cap B} a_n \\
& = \textstyle S(j_1) + \sum_{n \in (j_1,j] \cap C} a_n + \sum_{n \in (j_1,j] \cap (\N \setminus A)} a_n.
\end{align*}
By part $(iii)$ of our choice of $N$, $\sum_{n \in (j_1,j] \cap (A \setminus C)} a_n + \e > 0$; therefore
\begin{align*}
S(j) & < \textstyle S(j) + \left(\sum_{n \in (j_1,j] \cap (A \setminus C)} a_n + \e\right) \\
& = \textstyle S(j_1) + \sum_{n \in (j_1,j] \cap C} a_n + \sum_{n \in (j_1,j] \cap (\N \setminus A)} a_n + \sum_{n \in (j_1,j] \cap (A \setminus C)} a_n + \e \\
& = \textstyle S(j_1) + \sum_{n \in (j_1,j]} a_n + \e.
\end{align*}
Now $S(j_1) < \e$ by definition and $\sum_{n \in (j_1,j]} a_n < \e$ by part $(iv)$ of our choice of $N$. Thus $S(j) < 3\e$ as claimed.
\end{proof}

Hence for any $\e > 0$, there is some $N \in \N$ such that $-2\e < S(j) < 3\e$ for every $j \geq N$. It follows that $\sum_{n \in C \cup D}a_n = 0$, as claimed.
\end{proof}


\begin{proof}[Proof of the main theorem]
Let $\sum_{n \in \N}a^1_n$, $\sum_{n \in \N}a^2_n$, and $\sum_{n \in \N}a^3_n$ be conditionally convergent series. Let
$$\F = \textstyle \set{f_A}{A \sub \N \text{ is tame and } f_A \neq \0}.$$
If there is some $f_A \in \F$ such that $|\mathrm{dom}(f_A)| = 3$ then we are done, because the set $A$ sends all three of our series to infinity. So let us suppose that this is not the case. By Lemmas~\ref{lem:tameness} and \ref{lem:twotypes}, the family $\F$ must have one of the two types described in Lemma~\ref{lem:twotypes}.


\vspace{3mm}
\noindent \emph{Case 1:} Suppose that $\F$ is of Type $1$. For convenience, let us fix a single $\F \sub \mathrm{fn}(3,2)$ with Type $1$, as show below:

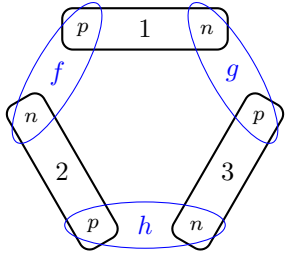
\begin{center}
\begin{tikzpicture}[xscale=1.1,yscale=1.1]

\draw[rounded corners,thick] (-1,3.206) rectangle (1,3.706) {};
\node at (-.75,3.45) {\scriptsize $p$};
\node at (.75,3.45) {\scriptsize $n$};
\draw[rotate=30,rounded corners,thick] (-.25,1) rectangle (.25,3) {};
\node at (-1.375,2.382) {\scriptsize $n$};
\node at (-.625,1.083) {\scriptsize $p$};
\node at (1.375,2.382) {\scriptsize $p$};
\node at (.625,1.083) {\scriptsize $n$};
\draw[rotate=-30,rounded corners,thick] (-.25,1) rectangle (.25,3) {};
\node at (0,3.456) {\small $1$};
\node at (-1,1.733) {\small $2$};
\node at (1,1.733) {\small $3$};
\node[blue] at (-1.06,2.92) {\small $f$};
\node[blue] at (1.06,2.92) {\small $g$};
\node[blue] at (0,1.083) {\small $h$};

\draw[blue] (0,1.083) ellipse (.97cm and .28cm);
\draw[rotate=-60,blue] (-2,2.382) ellipse (.97cm and .28cm);
\draw[rotate=60,blue] (2,2.382) ellipse (.97cm and .28cm);

\end{tikzpicture}
\end{center}

\noindent It suffices just to analyze this one instance of a Type 1 subset of $\mathrm{Fn}(3,2)$, because if $\F$ had any other equivalent type, then the argument below would remain essentially unchanged: we would merely need to relabel the series $\sum_{n \in \N}a^1_n$, $\sum_{n \in \N}a^2_n$, and $\sum_{n \in \N}a^3_n$, and/or interchange the role of positive and negative signs in one or more of the series.

\begin{claim}
There exists a partition of $\N$ into tame sets $F$, $G$, and $H$ such that
$f_F = f$, $f_G = g$, and $f_H = h$.
\end{claim}

\begin{proof}[Proof of claim]
First, partition $\N$ into $8$ sets $A_1, A_2, \dots, A_8$ according to where the terms of the $\sum_{n \in \N}a^i_n$ are positive or negative. For each $i \leq 8$, either $f_{A_i}$ is the empty function (this happens precisely when $\sum_{n \in A_i}a^j_n$ is absolutely convergent for all $j \in \{1,2,3\}$), or else $f_{A_i} \in \{f,g,h\}$.

Let
$J = \textstyle \bigcup \set{A_i}{f_{A_i} = \0},$
(this is a ``junk'' set) and then put
\begin{align*}
F & = \textstyle \bigcup \set{A_i}{f_{A_i} = f} \cup J,\\
G & = \textstyle \bigcup \set{A_i}{f_{A_i} = g} \text{, and} \\
H & = \textstyle \bigcup \set{A_i}{f_{A_i} = h}.
\end{align*}
By Lemma~\ref{lem:partition} and the fact that $\F$ has the type pictured above, there must be some $A_i$ with $f_{A_i} = f$, some $A_i$ with $f_{A_i} = g$, and some $A_i$ with $f_{A_i} = h$.
From this, and from Lemma \ref{lem:unionclosed}, it follows that each of $F$, $G$, and $H$ is tame, and that $f_F = f$, $f_G = g$, and $f_H = h$.
\end{proof}

Using Lemma~\ref{lem:balance1}, fix $B \sub F$ such that
$$\textstyle \sum_{n \in B}a^1_n = \infty \quad \text{and} \quad \sum_{n \in F \setminus B}a^1_n = \infty$$
and then set $A = B \cup G$. 

Because $F$, $G$, and $H$ form a partition of $\N$, and because $H$ is tame with $f_H = h$, the subseries $\sum_{F \cup G}a^1_n = \sum_{\N \setminus H}a^1_n$ is (conditionally) convergent. By Lemma~\ref{lem:disjoint}, 
this implies that
$$\textstyle \sum_{n \in B \cup G}a^1_n = \sum_{n \in (F \cup G) \setminus (F \setminus B)}a^1_n = \sum_{n \in F \cup G}a^1_n - \sum_{n \in F \setminus B}a^1_n= -\infty.$$

By Lemma~\ref{lem:subset}, $B$ is tame. Thus $f_B \in \F$, and because $\sum_{n \in B}a^1_n = \infty$ this implies $f_B = f$. In particular, $\sum_{n \in B}a^2_n = -\infty$. As $\sum_{n \in G}a^2_n$ converges absolutely, Lemma~\ref{lem:unionclosed} implies that $\sum_{n \in B \cup G}a^2_n = -\infty$.

Finally, $\sum_{n \in B}a^3_n$ converges absolutely (because $\sum_{n \in F}a^3_n$ converges absolutely and $B \sub F$), while $\sum_{n \in G}a^3_n = \infty$. Hence $\sum_{n \in B \cup G}a^3_n = \infty$ by Lemma~\ref{lem:unionclosed}.

\vspace{3mm}
\noindent \emph{Case 2:} Suppose that $\F$ is of Type $2$. For convenience, let us fix a single $\F \sub \mathrm{Fn}(3,2)$ with Type $2$, as show below:

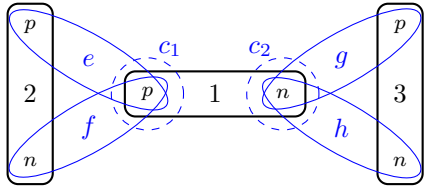
\vspace{1mm}
\begin{center}
\begin{tikzpicture}[xscale=1.2,yscale=1.2]

\draw[rounded corners,thick] (5.75,2) rectangle (7.75,2.5) {};
\node at (6,2.25) {\scriptsize $p$};
\node at (7.5,2.25) {\scriptsize $n$};
\draw[rounded corners,thick] (4.451,1.25) rectangle (4.951,3.25) {};
\node at (4.701,1.5) {\scriptsize $n$};
\node at (4.701,3) {\scriptsize $p$};
\draw[rounded corners,thick] (8.549,1.25) rectangle (9.049,3.25) {};
\node at (8.799,1.5) {\scriptsize $n$};
\node at (8.799,3) {\scriptsize $p$};
\node at (6.75,2.25) {\small $1$};
\node at (4.701,2.25) {\small $2$};
\node at (8.799,2.25) {\small $3$};
\node[blue] at (5.35,2.625) {\small $e$};
\node[blue] at (5.35,1.875) {\small $f$};
\node[blue] at (8.15,2.625) {\small $g$};
\node[blue] at (8.15,1.875) {\small $h$};
\node[blue] at (6.25,2.75) {\small $c_1$};
\node[blue] at (7.25,2.75) {\small $c_2$};

\draw[rotate=30,blue] (5.57,-1.042) ellipse (1cm and .28cm);
\draw[rotate=-30,blue] (3.318,4.95) ellipse (1cm and .28cm);
\draw[rotate=30,blue] (8.379,-1.792) ellipse (1cm and .28cm);
\draw[rotate=-30,blue] (6.123,5.7) ellipse (1cm and .28cm);
\draw[blue,dashed] (6,2.25) circle (.4cm);
\draw[blue,dashed] (7.5,2.25) circle (.4cm);

\end{tikzpicture}
\end{center}
\vspace{1mm}

\noindent As in case 1, it suffices just to analyze this one instance of a Type 1 subset of $\mathrm{Fn}(3,2)$, because if $\F$ had any other equivalent type, then the argument below would remain essentially unchanged.
\noindent We do not specify at this point whether either/both of the functions $c_1,c_2$ with domain $\{1\}$ are in $\F$.

\begin{claim}
There exists a partition of $\N$ into tame sets $E$, $F$, $G$, and $H$ such that
$f_E = e$, $f_F = f$, $f_G = g$, and $f_H = h$.
\end{claim}

\begin{proof}[Proof of claim]
First, partition $\N$ into $8$ sets $A_1, A_2, \dots, A_8$ according to where the terms of the $\sum_{n \in \N}a^i_n$ are positive or negative. For each $i \leq 8$, either $f_{A_i}$ is the empty function (this happens precisely when $\sum_{n \in A_i}a^j_n$ is absolutely convergent for all $j \in \{1,2,3\}$), or else $f_{A_i} \in \{c_1,c_2,e,f,g,h\}$.

Let
$J = \textstyle \bigcup \set{A_i}{f_{A_i} = \0},$
(this is a ``junk'' set) and then put
\begin{align*}
E & = \textstyle \bigcup \set{A_i}{f_{A_i} = e \text{ or } f_{A_i} = c_1} \cup J,\\
F & = \textstyle \bigcup \set{A_i}{f_{A_i} = f}, \\
G & = \textstyle \bigcup \set{A_i}{f_{A_i} = g \text{ or } f_{A_i} = c_2} \text{, and} \\
H & = \textstyle \bigcup \set{A_i}{f_{A_i} = h}.
\end{align*}
By Lemma~\ref{lem:partition} and the fact that $\F$ has the type pictured above, there must be some $A_i$ with $f_{A_i} = e$, some $A_i$ with $f_{A_i} = f$, some $A_i$ with $f_{A_i} = g$, and some $A_i$ with $f_{A_i} = h$.
From this, and from Lemma~\ref{lem:unionclosed}, it follows that each of $E$, $F$, $G$, and $H$ is tame, and that $f_E = e$, $f_F = f$, $f_G = g$, and $f_H = h$.
\end{proof}

The set $E \cup F$ is not tame with respect to $\sum_{n \in \N}a^2_n$, but by Lemma~\ref{lem:unionclosed} $E \cup F$ is tame with respect to $\sum_{n\in\N}a^1_n$. Similarly, $G \cup H$ is not tame with respect to $\sum_{n \in \N}a^3_n$, but is tame with respect to $\sum_{n\in\N}a^1_n$.
Thus, by Lemma~\ref{lem:balance2} there is some $C \sub G \cup H$ such that
$\textstyle \sum_{n \in E \cup C}a^1_n \text{ is converges}.$
We have three sub-cases to consider:

\vspace{3mm}
\noindent \emph{Case 2A:} Suppose that $\sum_{n \in C}a^3_n$ is absolutely convergent.

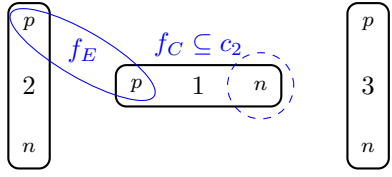
\vspace{1mm}
\begin{center}
\begin{tikzpicture}[xscale=1.1,yscale=1.1]

\draw[rounded corners,thick] (5.75,2) rectangle (7.75,2.5) {};
\node at (6,2.25) {\scriptsize $p$};
\node at (7.5,2.25) {\scriptsize $n$};
\draw[rounded corners,thick] (4.451,1.25) rectangle (4.951,3.25) {};
\node at (4.701,1.5) {\scriptsize $n$};
\node at (4.701,3) {\scriptsize $p$};
\draw[rounded corners,thick] (8.549,1.25) rectangle (9.049,3.25) {};
\node at (8.799,1.5) {\scriptsize $n$};
\node at (8.799,3) {\scriptsize $p$};
\node at (6.75,2.25) {\small $1$};
\node at (4.701,2.25) {\small $2$};
\node at (8.799,2.25) {\small $3$};
\node[blue] at (6.75,2.75) {\small $f_C \sub c_2$};

\node[blue] at (5.35,2.675) {\small $f_E$};
\draw[rotate=-30,blue] (3.318,4.95) ellipse (1cm and .28cm);
\draw[blue,dashed] (7.5,2.25) circle (.4cm);

\end{tikzpicture}
\end{center}
\vspace{1mm}

In this case, one may check that $C$ is tame and that either $f_C = \0$ or $f_C = c_2$. (This fact is not used below: we mention it only to aid the reader's intuition.) Take $A = E \cup C \cup G$.

Observe that $\sum_{n \in G \setminus C}a^3_n = \infty$ by Lemma~\ref{lem:unionclosed}. Because $G \setminus C$ is tame by Lemma~\ref{lem:subset}, $f_{G \setminus C} \in \F$, and so we must have $f_{G \setminus C} = f_G$. In particular, $\sum_{n \in G \setminus C}a^1_n = -\infty$.
Because $\sum_{E \cup C}a^1_n$ is conditionally convergent and $\sum_{n \in G \setminus C}a^1_n = -\infty$, Lemma~\ref{lem:disjoint} implies
$$\textstyle \sum_{n \in E \cup C \cup G}a^1_n = \sum_{n \in (E \cup C) \cup (G \setminus C)}a^1_n = \sum_{n \in E \cup C}a^1_n + \sum_{n \in G \setminus C}a^1_n = -\infty.$$

 $\sum_{G \cup C}a^2_n$ converges absolutely while $\sum_E a^2_n = \infty$, so $\sum_{n \in E \cup C \cup G}a^2_n = \infty$ by Lemma~\ref{lem:unionclosed}.

Finally, $\sum_{E \cup C}a^3_n$ converges and $\sum_{G \setminus C}a^3_n = \infty$, so by Lemma~\ref{lem:disjoint},
$$\textstyle \sum_{n \in E \cup C \cup G}a^3_n = \sum_{n \in E \cup C}a^3_n + \sum_{n \in G \setminus C}a^3_n = \infty.$$

\vspace{3mm}
\noindent \emph{Case 2B:} Suppose that $\sum_{n \in B}a^3_n$ is not absolutely convergent, but $B$ is tame with respect to $\sum_{n \in \N}a^3_n$.

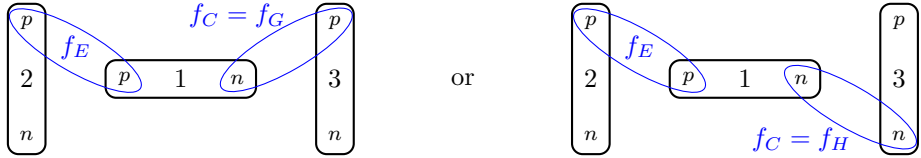
\vspace{1mm}
\begin{center}
\begin{tikzpicture}

\begin{scope}[shift={(-7.5,0)}]
\draw[rounded corners,thick] (5.75,2) rectangle (7.75,2.5) {};
\node at (6,2.25) {\scriptsize $p$};
\node at (7.5,2.25) {\scriptsize $n$};
\draw[rounded corners,thick] (4.451,1.25) rectangle (4.951,3.25) {};
\node at (4.701,1.5) {\scriptsize $n$};
\node at (4.701,3) {\scriptsize $p$};
\draw[rounded corners,thick] (8.549,1.25) rectangle (9.049,3.25) {};
\node at (8.799,1.5) {\scriptsize $n$};
\node at (8.799,3) {\scriptsize $p$};
\node at (6.75,2.25) {\small $1$};
\node at (4.701,2.25) {\small $2$};
\node at (8.799,2.25) {\small $3$};
\node[blue] at (7.5,3.1) {\small $f_C = f_G$};

\node[blue] at (5.35,2.675) {\small $f_E$};
\draw[rotate=-30,blue] (3.318,4.95) ellipse (1cm and .28cm);
\draw[rotate=30,blue] (8.379,-1.792) ellipse (1cm and .28cm);
\end{scope}

\node at (3,2.25) {\small or};

\draw[rounded corners,thick] (5.75,2) rectangle (7.75,2.5) {};
\node at (6,2.25) {\scriptsize $p$};
\node at (7.5,2.25) {\scriptsize $n$};
\draw[rounded corners,thick] (4.451,1.25) rectangle (4.951,3.25) {};
\node at (4.701,1.5) {\scriptsize $n$};
\node at (4.701,3) {\scriptsize $p$};
\draw[rounded corners,thick] (8.549,1.25) rectangle (9.049,3.25) {};
\node at (8.799,1.5) {\scriptsize $n$};
\node at (8.799,3) {\scriptsize $p$};
\node at (6.75,2.25) {\small $1$};
\node at (4.701,2.25) {\small $2$};
\node at (8.799,2.25) {\small $3$};
\node[blue] at (7.5,1.45) {\small $f_C = f_H$};

\node[blue] at (5.35,2.675) {\small $f_E$};
\draw[rotate=-30,blue] (3.318,4.95) ellipse (1cm and .28cm);
\draw[rotate=-30,blue] (6.123,5.7) ellipse (1cm and .28cm);

\end{tikzpicture}
\end{center}
\vspace{1mm}

$C$ is tame with respect to $\sum_{n \in \N}a^1_n$ by Lemma~\ref{lem:subset}, because $C \sub G \cup H$ and $G \cup H$ is tame with respect to $\sum_{n \in \N}a^1_n$. $C$ is tame with respect to $\sum_{n \in \N}a^2_n$ because $\sum_{n \in C}a^1_n$ is absolutely convergent, and $C$ is tame with respect to $\sum_{n \in \N}a^3_n$ by hypothesis. Hence $C$ is tame and $f_C \in \F$.
Because $\sum_{n \in \N}a^3_n$ is not absolutely convergent, $f_C \in \F$ implies that either $f_C = f_G$ or $f_C = f_H$. Let us suppose that $f_C = f_G$. (The argument is essentially identical if instead $f_C = f_H$.) In particular, $\sum_{n \in C}a^3_n = \infty$ and $\sum_{n \in C}a^1_n = -\infty$.

Using Lemma~\ref{lem:balance1}, fix $B \sub C$ such that
$$\textstyle \sum_{n \in B}a^3_n = \infty \quad \text{and} \quad \sum_{n \in C \setminus B}a^3_n = \infty$$
and then set $A = E \cup B$. Both $B$ and $C \setminus B$ are tame by Lemma~\ref{lem:subset}, so (as in the previous paragraph) $f_B = f_{C \setminus B} = f_G$.

Because $\sum_{E \cup C}a^1_n$ converges and $\sum_{n \in C \setminus B}a^1_n = -\infty$ (because $f_{C \setminus B} = f_G$), Lemma~\ref{lem:disjoint} implies that
$$\textstyle \sum_{n \in E \cup B}a^1_n = \sum_{n \in (E \cup C) \setminus (C \setminus B)}a^1_n = \sum_{n \in E \cup C}a^1_n - \sum_{n \in C \setminus B}a^1_n = \infty.$$

Because $\sum_{n \in E}a^2_n = \infty$ and $\sum_{n \in B}a^2_n$ is absolutely convergent, we have $\sum_{n \in E \cup B}a^2_n = \infty$.

Finally, because $\sum_{n \in E}a^3_n$ is absolutely convergent and $\sum_{n \in B}a^3_n = \infty$, we have $\sum_{n \in E \cup B}a^3_n = \infty$.

\vspace{3mm}
\noindent \emph{Case 2C:} Suppose that $C$ is not tame with respect to $\sum_{n \in \N}a^3_n$.

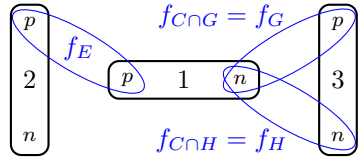
\vspace{1mm}
\begin{center}
\begin{tikzpicture}

\draw[rounded corners,thick] (5.75,2) rectangle (7.75,2.5) {};
\node at (6,2.25) {\scriptsize $p$};
\node at (7.5,2.25) {\scriptsize $n$};
\draw[rounded corners,thick] (4.451,1.25) rectangle (4.951,3.25) {};
\node at (4.701,1.5) {\scriptsize $n$};
\node at (4.701,3) {\scriptsize $p$};
\draw[rounded corners,thick] (8.549,1.25) rectangle (9.049,3.25) {};
\node at (8.799,1.5) {\scriptsize $n$};
\node at (8.799,3) {\scriptsize $p$};
\node at (6.75,2.25) {\small $1$};
\node at (4.701,2.25) {\small $2$};
\node at (8.799,2.25) {\small $3$};
\node[blue] at (7.25,3.1) {\small $f_{C \cap G} = f_G$};
\node[blue] at (7.25,1.45) {\small $f_{C \cap H} = f_H$};

\node[blue] at (5.35,2.675) {\small $f_E$};
\draw[rotate=-30,blue] (3.318,4.95) ellipse (1cm and .28cm);
\draw[rotate=-30,blue] (6.123,5.7) ellipse (1cm and .28cm);
\draw[rotate=30,blue] (8.379,-1.792) ellipse (1cm and .28cm);

\end{tikzpicture}
\end{center}
\vspace{1mm}

In this case, $\sum_{n \in C \cap G}a^3_n = \infty$ and $\sum_{n \in C \cap H}a^3_n = -\infty$. As both $C \cap G$ and $C \cap H$ are tame by Lemma~\ref{lem:subset}, $f_{C \cap G} = f_G$ and $f_{C \cap H} = f_H$ as shown in the picture. Let $A = E \cup (C \cap G)$. 

Because $\sum_{E \cup C}a^1_n$ is conditionally convergent and $\textstyle \sum_{n \in C \cap H}a^1_n = -\infty$, Lemma~\ref{lem:disjoint} implies that
$$\textstyle \sum_{n \in E \cup (C \cap G)}a^1_n = \sum_{n \in (E \cup C) \setminus (C \cap H)}a^1_n = \sum_{n \in E \cup C}a^1_n - \sum_{n \in C \cap H}a^1_n = \infty.$$

Because $\sum_{n \in E}a^2_n = \infty$ and $\sum_{n \in C \cap G}a^2_n$ is absolutely convergent, we have $\sum_{n \in E \cup (C \cap G)}a^2_n = \infty$.

Finally, because $\sum_{n \in E}a^3_n$ is absolutely convergent and $\sum_{n \in C \cap G}a^3_n = \infty$, we have $\sum_{n \in E \cup (C \cap G)}a^3_n = \infty$.
\end{proof}



\section{A counterexample for four series}\label{sec:example}

This section contains an example of four conditionally convergent series such that no $A \sub \N$ sends all four series to infinity. The example is due to Fedor Nazarov, and was posted by him to the online forum MathOverflow \cite{fedja} in response to a question posted there by the author. It is reproduced here with his permission.

To construct this example, begin by partitioning $\N$ into adjacent intervals $I_1, I_2, I_3, \dots$ of increasing length. Let $b_m$ denote the length of $I_m$. We do not specify at this point in the proof exactly how fast the function $m \mapsto b_m$ increases; let us just say for now that it is some rapidly increasing function. For convenience, we shall take each $b_m$ to be an even number, so that the first member of every $I_m$ is an odd number.

We now define our collection of $4$ series by specifying the terms of each series on each of the disjoint intervals $I_m$. Let us denote the four series by $\sum_{n \in \N} a^1_n$, $\sum_{n \in \N} a^2_n$, $\sum_{n \in \N} a^3_n$, and $\sum_{n \in \N} a^4_n$. For each $m \in \N$ and $i \in \{1,2,3,4\}$, define $a^i_n$ for all $n \in I_m$ as follows:
\begin{itemize}

\item If $m$ is odd and $n \in I_m$, then 
\begin{itemize}
\item[$\circ$] $a_n^1 = a^2_n = \frac{1}{m}$ for odd $n$ and $a_n^1 = a^2_n = -\frac{1}{m}$ for even $n$,
\item[$\circ$] $a_n^3 = \frac{1}{b_m}$ for odd $n$ and $a_n^3 = -\frac{1}{b_m}$ for even $n$,
\item[$\circ$] $a_n^4 = 0$ for all $n$.
\end{itemize}

\item If $m$ is even and $n \in I_m$, then 
\begin{itemize}
\item[$\circ$] $a_n^1 = -a^2_n = \frac{1}{m}$ for odd $n$ and $a_n^1 = -a^2_n = -\frac{1}{m}$ for even $n$,
\item[$\circ$] $a_n^3 = 0$ for all $n$.
\item[$\circ$] $a_n^4 = \frac{1}{b_m}$ for odd $n$ and $a_n^4 = -\frac{1}{b_m}$ for even $n$,
\end{itemize}

\end{itemize}
Explicitly, our $4$ series look like this on $I_m$ when $m$ is odd:

\vspace{-4mm}
\renewcommand*{\arraystretch}{2}
$$\begin{matrix}

\text{ series } 1: \ \, & \textstyle + \frac{1}{m} \ & \textstyle - \frac{1}{m} \ & \textstyle + \frac{1}{m} \ & \textstyle - \frac{1}{m} \ & \textstyle + \frac{1}{m} \ & \textstyle - \frac{1}{m} \ & \textstyle + \frac{1}{m} \ & \textstyle - \frac{1}{m} & \ \dots \\

\text{ series } 2: \ \, & \textstyle + \frac{1}{m} \ & \textstyle - \frac{1}{m} \ & \textstyle + \frac{1}{m} \ & \textstyle - \frac{1}{m} \ & \textstyle + \frac{1}{m} \ & \textstyle - \frac{1}{m} \ & \textstyle + \frac{1}{m} \ & \textstyle - \frac{1}{m} & \ \dots  \\

\text{ series } 3: \ \, & \textstyle + \frac{1}{b_m} \, & \textstyle - \frac{1}{b_m} \, & \textstyle + \frac{1}{b_m} \, & \textstyle - \frac{1}{b_m} \, & \textstyle + \frac{1}{b_m} \, & \textstyle - \frac{1}{b_m} \, & \textstyle + \frac{1}{b_m} \, & \textstyle - \frac{1}{b_m} & \ \dots  \\

\text{ series } 4: \ \, & \textstyle + 0 \ & \textstyle + 0 \ & \textstyle + 0 \ & \textstyle + 0 \ & \textstyle + 0 \ & \textstyle + 0 \ & \textstyle + 0 \ & \textstyle + 0 & \ \dots 

\end{matrix}$$
\vspace{-.5mm}

\noindent And our $4$ series look like this on $I_m$ when $m$ is even:

\vspace{-4mm}
\renewcommand*{\arraystretch}{2}
$$\begin{matrix}

\text{ series } 1: \ \, & \textstyle + \frac{1}{m} \ & \textstyle - \frac{1}{m} \ & \textstyle + \frac{1}{m} \ & \textstyle - \frac{1}{m} \ & \textstyle + \frac{1}{m} \ & \textstyle - \frac{1}{m} \ & \textstyle + \frac{1}{m} \ & \textstyle - \frac{1}{m} & \ \dots \\

\text{ series } 2: \ \, & \textstyle - \frac{1}{m} \ & \textstyle + \frac{1}{m} \ & \textstyle - \frac{1}{m} \ & \textstyle + \frac{1}{m} \ & \textstyle - \frac{1}{m} \ & \textstyle + \frac{1}{m} \ & \textstyle - \frac{1}{m} \ & \textstyle + \frac{1}{m} & \ \dots  \\

\text{ series } 3: \ \, & \textstyle + 0 \ & \textstyle + 0 \ & \textstyle + 0 \ & \textstyle + 0 \ & \textstyle + 0 \ & \textstyle + 0 \ & \textstyle + 0 \ & \textstyle + 0 & \ \dots  \\

\text{ series } 4: \ \, & \textstyle + \frac{1}{b_m} \, & \textstyle - \frac{1}{b_m} \, & \textstyle + \frac{1}{b_m} \, & \textstyle - \frac{1}{b_m} \, & \textstyle + \frac{1}{b_m} \, & \textstyle - \frac{1}{b_m} \, & \textstyle + \frac{1}{b_m} \, & \textstyle - \frac{1}{b_m} & \ \dots 

\end{matrix}$$
\vspace{-.5mm}

Assuming that $\lim_{m \to \infty} b_m = \infty$, it is clear that each of these series converges conditionally to $0$.

Before proceeding with a detailed proof of why these $4$ series have the stated property, we describe the idea behind it; this paragraph can be omitted by readers who just want the detailed proof. In each odd block, our first two series match, and so the terms from odd blocks can be used to send the first two series either both to $+\infty$ or both to $-\infty$, but they are not useful for sending one to $\infty$ and the other to $-\infty$. Similarly, the terms from even blocks can be used to send one of the first two series to $\infty$ and the other to $-\infty$, but they are not useful for sending either both to $+\infty$ or both to $-\infty$. However, if the $b_m$ grow fast enough, then sending our third series to infinity requires us to include many terms from odd blocks, and sending our fourth series to infinity requires us to include many terms from even blocks. This tension -- the odd blocks trying to send our first two series to matching infinities and our even blocks trying to do the opposite -- ultimately results in the conclusion that if both series $3$ and $4$ go to infinity, then one of our first two series diverges by oscillation.

Let us now make this argument rigorous. Let $b_1, b_2, \dots, b_m, \dots$ be an increasing sequence of even numbers, with $b_1 = 2$, satisfying the following recurrence relation:
$$b_{m+1} \, \geq \, m^3 (1 + b_1 + b_2 + \dots + b_m)$$
Consider the four series defined above in terms of the $b_m$, let $A \sub \N$, and suppose that $A$ sends the third and fourth series to infinity. We shall show that either the first or the second series diverges by oscillation.

More specifically, let us suppose that $\sum_{n \in A}a^3_n = \sum_{n \in A}a^4_n = \infty$, and then show that this implies $\sum_{n \in A}a^2_n$ diverges by oscillation. 

We include only this one case, because the other three cases are handled in exactly the same way. For example, if instead of supposing that $\sum_{n \in A}a^3_n = \sum_{n \in A}a^4_n = \infty$ we were to suppose that $\sum_{n \in A}a^3_n = \sum_{n \in A}a^4_n = -\infty$, then a similar argument would show that again $\sum_{n \in A}a^2_n$ diverges by oscillation. If we were to suppose either that $\sum_{n \in A}a^3_n = \infty$ and $\sum_{n \in A}a^4_n = -\infty$, or that $\sum_{n \in A}a^3_n = -\infty$ and $\sum_{n \in A}a^4_n = \infty$, then a similar argument would show that $\sum_{n \in A}a^1_n$ diverges by oscillation.

For each $m$, let $\Delta(m)$ denote the imbalance of positive terms over negative terms in $A$ from block $m$:
\begin{align*}
\Delta(m) \,=\, & \card{\set{n \in A \cap I_m}{n \text{ is odd}}} - \card{\set{n \in A \cap I_m}{n \text{ is even}}}.
\end{align*}
Observe, in our third subseries $\sum_{n \in A}a^3_n= \sum_{m \in \N}\left(\sum_{n \in A \cap I_m}a^3_n\right)$, that
\begin{align*}
\textstyle \sum_{n \in A \cap I_m} a_n^3 =
\begin{cases}
\frac{\Delta(m)}{b_m} & \text{ if } m \text{ is odd} \\
0 & \text{ if } m \text{ is even.}
\end{cases}
\end{align*}
It follows that $\Delta(m) > \nicefrac{b_m}{m^2}$ for infinitely many odd $m$ because, if not, then the third subseries cannot grow fast enough to sum to $\infty$. More precisely, if there were some $M$ such that $\Delta(m) \leq \nicefrac{b_m}{m^2}$ for every odd $m \geq M$, then
\begin{align*}
\textstyle \sum_{n \in \N}a^3_n & = \textstyle \sum_{m \geq M} \left( \sum_{n \in I_m \cap A} a_n^3 \right) \\
& \leq \textstyle  \, \sum_{m \geq M,\, m \text{ odd}} \frac{\Delta(m)}{b_m} \, \leq \, \sum_{m \geq M} \frac{1}{m^2} \, < \, \infty,
\end{align*}
contradicting the assumption that $\sum_{n \in A} a^3_n = \infty$. Thus, for infinitely many odd $m$,
$$\Delta(m) \, > \, \frac{b_m}{m^2} \, \geq \, m(1 + b_1+b_2+\dots+b_{m-1}).$$

Now consider the second subseries $\sum_{n \in A}a^2_n= \sum_{m \in \N}\left(\sum_{n \in A \cap I_m}a^2_n\right)$. If $m$ is odd and if $\Delta(m) \, > \, m(1 + b_1+b_2+\dots+b_{m-1})$, then
$$\textstyle \sum_{n \in I_m,\, n \in A} a^2_n \, = \, \frac{-\Delta(m)}{m} \, < \, -1 - b_1-b_2-\dots-b_{m-1},$$
and this is greater in absolute value than all the preceding terms of the subseries combined:
\begin{align*}
\textstyle \card{\sum_{n < \min I_m,\, n \in A} a^2_n} \, & \leq \, \textstyle \sum_{j < m}\card{\sum_{n \in A \cap I_j}a^2_n} \\
& \leq \, \textstyle \sum_{j < m} \frac{b_j}{j} \, < \, b_1+b_2+\dots+b_{m-1}.
\end{align*}
It follows that
$$\textstyle \sum_{n \leq \max I_m,\, n \in A} a^2_n \,<\, -1.$$
This holds for infinitely many $m$, so in particular, the finite partial sum
$$\textstyle \sum_{n \in A,\, n \leq k} a^2_n \,<\, -1$$
for infinitely many $k \in \N$.

By considering the even blocks instead of the odd blocks, and by using the assumption that $\sum_{n \in A} a^4_n = \infty$, a similar argument shows that
$$\textstyle \sum_{n \in A,\, n \leq k} a^2_n \,>\, 1$$
for infinitely many $k \in \N$. Hence $\sum_{n \in A}a^2_n$ diverges by oscillation as claimed.

\end{document}